\documentclass[12pt,a4paper,oneside]{amsart} 
\usepackage{amsmath, amssymb, amsthm,amsfonts,mathtools,microtype,cite}
\usepackage{xcolor}
\definecolor{linkblue}{RGB}{1,1,190}
\definecolor{citered}{RGB}{190,1,1}
\definecolor{recursionblue}{RGB}{70,130,180}

\usepackage[linkcolor=linkblue
,urlcolor=linkblue
,citecolor=citered
,colorlinks
,bookmarksopen=true,
]{hyperref}

\makeatletter
\AtBeginDocument{
	\hypersetup{
		pdftitle  = {\@title},
	}
}
\makeatother

\title[Null polynomials over a finite ring]{Null polynomials over a finite ring need not form a two-sided ideal}
\author{Valentin Havlovec}
\address{Institute of Analysis and Number Theory, Graz University of Technology,
	Kopernikusgasse 24, 8010 Graz, Austria}
\email{havlovec@math.tugraz.at}
\thanks{This research was funded in whole or in part by the Austrian Science Fund (FWF) [10.55776/P35788].}

\newcommand{\Ftwo}{\mathbb{F}_2}
\newcommand{\T}{\mathrm{T}}
\newcommand{\M}{\mathrm{M}}
\newcommand{\TtwoFtwo}{\T_2(\Ftwo)}
\newcommand{\MtwoFtwo}{\M_2(\Ftwo)}
\newcommand{\MfourFtwo}{{\M}_4(\Ftwo)}

\newcommand{\TnFtwo}{{\T}_n(\Ftwo)}
\newcommand{\ZZ}{\mathbb{Z}}

\DeclareMathOperator{\N}{N}

\newtheorem{thm}{Theorem}[section]
\theoremstyle{definition}
\newtheorem{rem}[thm]{Remark}

\begin{document}
	\begin{abstract}
		We give an example of a finite ring for which the set of polynomials inducing the zero function fails to be a two-sided ideal, disproving a conjecture of Werner.
	\end{abstract}
	\maketitle
	\section{Introduction}
	Let $R$ be a unital ring, and let $R[x]$ denote the polynomial ring over $R$ in a central indeterminate $x$. Let $\N(R) = \{f \in R[x] \mid f(r) = 0 \text{ for all } r \in R\}$ be the set of null polynomials with coefficients in $R$. Here we use the convention of right substitution, i.e., for $f = \sum_{i=0}^{n} a_i x^i$ and $r \in R$, we define $f(r) = \sum_{i=0}^{n} a_i r^i$.
	
	If $R$ is a commutative ring, then $\N(R)$ is easily seen to be an ideal of $R[x]$. However, if $R$ is noncommutative, whether this holds is not clear, since evaluation at $r$ fails to preserve multiplication in general. Nevertheless, the set $\N(R)$ is always a left ideal.
	Indeed, it is clearly an additive subgroup of $R[x]$. Moreover, if $g = \sum_j b_j x^j$ is an arbitrary polynomial and $f\in \N(R)$, then for all $r \in R$ we have $(gf)(r) = \sum_{j}b_j f(r)r^j =0$. Werner proved that $\N(R)$ is a two-sided ideal for a large class of rings, including local rings, semisimple rings, matrix rings over arbitrary commutative rings, and rings of odd order~\cite[Theorem 3.7]{Werner_PolynomialsThatKill}, and conjectured that this holds for all finite rings.
	
	Further evidence for this conjecture was given by Frisch, who showed that $\N(R)$ is a two-sided ideal when $R$ is the ring of upper triangular matrices over a commutative ring \cite[Theorem 5.2]{Frisch_PolFunOnUpperTriangularMatrixAlgebras}. More generally, the same conclusion holds when $R$ is a structural matrix ring \cite{Havlovec_IVPStructuralMatrixRings}.
	
	We disprove the conjecture by giving an example of a finite ring whose null polynomials do not form a two-sided ideal.

	\section{The counterexample}
	Throughout, $\mathrm M_n(\Ftwo)$ denotes the ring of $n\times n$ matrices over $\Ftwo$, and $\TnFtwo$ its subring of upper triangular matrices.
	\begin{thm}
		Let
		\[
		R = \left\{ \begin{pmatrix}
			A & B \\
			0 & A
		\end{pmatrix}  \mathrel{\Big|} A \in \TtwoFtwo \text{ and } B \in \MtwoFtwo \right\}.
		\]
		Then $R$ is a finite ring and $\N(R)$ is not a right ideal of $R[x]$. Consequently, $\N(R)$ is not a two-sided ideal of $R[x]$.
	\end{thm}
	
	\begin{proof}
		The set $R$ is an additive subgroup of $\MfourFtwo$ and contains the identity matrix. Moreover, 
		\[\begin{pmatrix}
			A & B\\
			0 & A
		\end{pmatrix}
		\begin{pmatrix}
			C & D\\
			0 & C
		\end{pmatrix} = 
		\begin{pmatrix}
			AC & AD + BC\\
			0 & AC
		\end{pmatrix} \in R.
		\] Thus
		$R$ is a unital subring of $\MfourFtwo$. Equivalently, $R$ is the trivial extension of $\TtwoFtwo$ by the bimodule $\MtwoFtwo$.
		
		Let $E = \begin{psmallmatrix}
			1&0\\
			0&0
		\end{psmallmatrix}\in \TtwoFtwo$, and $G = \begin{psmallmatrix}
			0&1\\
			0&0
		\end{psmallmatrix}\in \TtwoFtwo$. Further, let $P=\begin{psmallmatrix}
			E&0\\
			0&E
		\end{psmallmatrix}\in R$, and $Q = \begin{psmallmatrix}
			G&0\\
			0&G
		\end{psmallmatrix}\in R$.
		We claim that the polynomial
		\[
		f = Qx^2 + Px^3 + (P+Q)x^4 + Px^5 + Px^6
		\]
		satisfies $f \in \N(R)$ but $fP \notin \N(R)$.
		
		We first show that \(f\in\N(R)\). Let $r\in R$ be arbitrary and observe that
		\[
		f = (Q + Px + Px^2)(x^2 + x^4).
		\]		
		Let $v = Q + Px + Px^2$ and $w = x^2 + x^4$.
		Although evaluation is not multiplicative in general, in this instance we do have $f(r) = v(r)w(r)$ since the coefficients of $w$ are central.		
		Now let 
		\[r=
		\begin{pmatrix}
			A&B\\
			0&A
		\end{pmatrix}
		\in R,
		\qquad
		A=
		\begin{pmatrix}
			a&b\\
			0&c
		\end{pmatrix}.
		\]
		Then
		\[
		r^2=
		\begin{pmatrix}
			A^2&AB+BA\\
			0&A^2
		\end{pmatrix}.
		\]
		Since \(a,b,c\in\Ftwo\), we have
		\[
		A^2=
		\begin{pmatrix}
			a&b(a+c)\\
			0&c
		\end{pmatrix}.
		\]
		Thus
		\[
		A^2=A
		\quad\Longleftrightarrow\quad
		b=0\ \text{or}\ a\neq c.
		\]
		Consequently, either \(A^2=A\), or \(b=1\) and \(a=c\). In the latter
		case,
		\[
		A=\varepsilon I_2+G
		\qquad\text{for some }\varepsilon\in\Ftwo,
		\]
		where $I_2$ denotes the $2 \times 2$ identity matrix.
		
		\begin{samepage}
		Suppose first that \(A^2=A\). In this case,
		\[
			r^2 = \begin{pmatrix}
				A&AB+BA\\
				0&A
			\end{pmatrix},
		\]		
		and
		\begin{align*}	
			r^4&=\begin{pmatrix}
				A&AB+ABA+ABA+BA\\
				0&A
			\end{pmatrix}\\
			&=\begin{pmatrix}
				A&AB+BA\\
				0&A
			\end{pmatrix}\\
			&=r^2,
		\end{align*}
		where we have used that the characteristic is~$2$. 
		Thus $w(r) = r^2+r^4=0$, which implies $f(r)=0$.
	\end{samepage}
	
		Now suppose that $A=\varepsilon I_2+G$. Since $G^2=0$, we have $A^2=\varepsilon I_2$ and hence
		\[
		r^2=
		\begin{pmatrix}
			\varepsilon I_2&AB+BA\\
			0&\varepsilon I_2
		\end{pmatrix},
		\quad\text{and}\quad
		r^4=
		\begin{pmatrix}
			\varepsilon I_2&0\\
			0&\varepsilon I_2
		\end{pmatrix}.
		\]
		Consequently,
		\[
		w(r)=r^2+r^4=
		\begin{pmatrix}
			0&AB+BA\\
			0&0
		\end{pmatrix}.
		\]
		Recall that because $w$ has central coefficients,
		\[
		f(r)=Qw(r)+Prw(r)+Pr^2w(r).
		\]
		Setting $C:=AB+BA$, a block-matrix multiplication gives
		\[
		f(r)
		=
		\begin{pmatrix}
			0&GC+EAC+\varepsilon EC\\
			0&0
		\end{pmatrix}
		=
		\begin{pmatrix}
			0&GC+EGC\\
			0&0
		\end{pmatrix},
		\]
		because $A=\varepsilon I_2+G$. Finally, since $EG=G$, the upper-right block is
		\[
		GC+EGC=GC+GC=0.
		\]
		Thus $f(r)=0$ also in this case. Since $r\in R$ was arbitrary, we conclude that $f\in\N(R)$.
		
		It remains to show that $fP \notin \N(R)$. Using $P^2 = P$ and $QP = 0$, we obtain $fP = Px^3+Px^4+Px^5+Px^6$. Now let $J = \begin{psmallmatrix}
			0&1&0&0\\
			0&0&1&0\\
			0&0&0&1\\
			0&0&0&0
		\end{psmallmatrix} \in R$. Then $J^3 = \begin{psmallmatrix}0&0&0&1\\0&0&0&0\\0&0&0&0\\0&0&0&0\end{psmallmatrix}$, while all higher powers vanish. Hence
		\[
		(fP)(J) = \begin{pmatrix}
			1&0&0&0\\
			0&0&0&0\\
			0&0&1&0\\
			0&0&0&0
		\end{pmatrix}\begin{pmatrix}
			0&0&0&1\\
			0&0&0&0\\
			0&0&0&0\\
			0&0&0&0
		\end{pmatrix} = \begin{pmatrix}
			0&0&0&1\\
			0&0&0&0\\
			0&0&0&0\\
			0&0&0&0
		\end{pmatrix} \neq 0.
		\]
		This shows that $\N(R)$ is not closed under right multiplication by $P$, so in particular $\N(R)$ is not a right ideal of $R[x]$.
	\end{proof}
	
	\begin{rem}
		Although $R$ has a prescribed zero lower-left block, it is not a
		structural matrix ring: membership in $R$ also requires its two
		diagonal blocks to be equal, a condition that cannot be expressed
		solely by prescribing certain entries to be zero. Hence the known
		results for structural matrix rings do not apply here.
	\end{rem}
	
	\begin{rem}
		This example can also be used to construct a finitely generated torsion-free $\mathbb{Z}$-algebra for which the set of integer-valued polynomials does not form a ring. This provides the exampled requested in \cite[Problem 27a]{CahenFontanaFrischGlaz_OpenProblems}, and disproves the conjecture in \cite[Problem 27b]{CahenFontanaFrischGlaz_OpenProblems}.
		
		Indeed, replace $\Ftwo$ with $\ZZ$ in the definition of the ring $R$ to obtain the module-finite, free $\ZZ$-algebra $S$. Then $R \cong S/2S$. By \cite[Theorem 2.4]{Werner_PolynomialsThatKill}, the failure of $\N(R)$ to be a two-sided ideal implies that the integer-valued polynomials over $S$ do not form a ring.
	\end{rem}
	
	\section*{Declaration on the use of LLMs}
	The counterexample and an initial proof were generated by GPT-5.6 Sol. The author subsequently verified the construction independently and simplified the proof. The author takes full responsibility for the mathematical content of the article.

	\bibliographystyle{amsplain-doi}
	\bibliography{bibliography}
\end{document}